\RequirePackage{fix-cm}
\documentclass[smallextended]{svjour3}       
\smartqed  
\usepackage{graphicx}
\usepackage{amsmath}
\usepackage{amssymb}
\usepackage{amsfonts}
\usepackage{latexsym}
\usepackage{url}
\usepackage{color}
\usepackage{graphicx}
\newcommand{\E}{\operatorname{E}}
\newcommand{\F}{\operatorname{F}}
\newcommand{\G}{\operatorname{G}}
\newcommand{\Ha}{\operatorname{H}}

\newcommand{\Prob}{\operatorname{P}}
\newcommand{\sign}{\operatorname{sign}}

\begin{document}
\allowdisplaybreaks

\title{On consistency of the likelihood moment estimators for a linear process with regularly varying innovations}


\author{Lukas Martig       \and
        J\"urg H\"usler 
}


\institute{L. Martig \at
              Sidlerstrasse 5, 3012 Bern, Switzerland \\
              Tel.: +41-31-6318811\\
              Fax: +41-31-6313870\\
              \email{lukasmartig@gmail.com}           
           \and
           J. H\"usler \at
           Sidlerstrasse 5, 3012 Bern, Switzerland
}

\date{Received: 8 July 2015 / Accepted: date}

\maketitle

\begin{abstract}
In 1975 James Pickands III showed that the excesses over a high threshold are approximatly Generalized Pareto distributed. Since then, a variety of estimators for the parameters of this cdf have been studied, but always assuming the underlying data to be independent. In this paper we consider the special case where the underlying data arises from a linear process with regularly varying (i.e. heavy-tailed) innovations. Using this setup, we then show that the likelihood moment estimators introduced by Zhang (2007) are consistent estimators for the parameters of the Generalized Pareto distribution.
\keywords{Generalized Pareto distribution \and Linear processes \and Heavy-tailed data \and Likelihood moment estimators \and Consistency}
\subclass{60G50 \and 60G70 \and 62G32}
\end{abstract}

\section{Introduction}
\setcounter{equation}{0}
\renewcommand{\theequation}{\thesection.\arabic{equation}}
A random variable $X$ has a heavy right or left tail if there exists a positive parameter $\gamma$, called extremal index, such that as $x \to \infty$:
\begin{equation}
\Prob(X > x)\sim x^{-1/\gamma}L(x) \hspace{0.5cm}  \text{or} \hspace{0.5cm} \Prob(X < -x) \sim x^{-1/\gamma}L(x), \label{one}
\end{equation}
where $L(x)$ is a slowly varying function satisfying $L(yx) \sim L(x)$ for any $y > 0$. Here we use the common notation $g(x) \sim f(x)$ as $x \to \infty$ for $\lim_{x \to \infty} g(x)/f(x) = 1$. Since heavy-tail analysis is a special branch of extreme value theory, there are basically two main approaches for modeling extreme events:
\newline The first approach -- the more traditional one introduced by Fisher \& Tippet (1928, \cite{Fisher}) -- is based on the Generalized Extreme Value distribution (GEVD) $ \G _\gamma(x) := \exp\left( - (1+\gamma x)^{-1/\gamma} \right)$, $1 + \gamma x > 0$, and is used to model block extremes, i.e. maxima (or minima) observed within a pre-defined period. Actually, as $\gamma$ is assumed to be positive in (\ref{one}), the GEVD simply reduces to the Fr\'echet distribution in this special setup.
\newline The second and more recent approach -- the one being discussed in this paper -- is to model exceedances over high thresholds. This kind of modeling was introduced by Pickands (1975, \cite{Pickands}) and is based on the Generalized Pareto distribution (GPD) $\Ha_{\gamma, \sigma^*}(x) := 1- \left(1+ \gamma x / \sigma^* \right)^{-1/\gamma}$, where $1+ \gamma x / \sigma^*  > 0$ and $\sigma^*  = \sigma^*(t)$ is a suitable scale function. Notice that as before, the GPD reduces to the Pareto distribution in case $\gamma > 0$.
\vspace{0.25cm}
\newline Classical studies on how to fit a GPD (see section 2 for references) usually cover the simplest case when the underlying data are i.i.d. Though the assumption of independence has its practical advantages, most collected data sets have non-negligible dependence structures which need to be included in statistical modeling. An easy but quite powerful way to model such a dependence is the use of so-called autoregressive moving average (ARMA) time series models. In detail, a stationary time series $\{X_n \}_{n \in \mathbb{Z}_0}$ is called an ARMA($p$,$q$)-process if it has additive innovations $\{Z_n \}_{n \in \mathbb{Z}_0}$ satisfying the following recurrence formula:
\vspace{0.2cm}
\begin{equation}
X_n - \phi_1 X_{n-1} - \ldots - \phi_p X_{n-p} = Z_n + \theta_1 Z_{n-1} + \ldots + \theta_q Z_{n-q}, \label{ARMAintro}
\end{equation}
\vspace{-0.3cm}
\newline with $\phi_1, \ldots, \phi_p; \theta_1, \ldots, \theta_q \in \mathbb{R}$ and $p,q \in \mathbb{N}_0$. From \cite{Brockwell}, Theorem 3.1.1, it then follows that if $\phi(z)$ has no root on the unit circle and $\phi(z)$ and $\theta(z)$ have no common zeroes, the ARMA process in (\ref{ARMAintro}) has the causal representation
\begin{equation}
X_n = \sum_{j=0}^{\infty} c_j Z_{n-j}, \label{causalintro}
\end{equation}
where the coefficients $\{c_j\}_{j \geq 0}$ are determined by the relation
\begin{equation*}
c(z) = \sum_{j=0}^{\infty}c_j z^j = \theta(z)/\phi(z), \hspace{0.5cm} |z| \leq 1.
\end{equation*}
\vspace{-0.3cm}
\newline
\cite{Datta} studied the asymptotic behavior of linear processes (also known as MA($\infty$)-processes) given in (\ref{causalintro}) for the case where the innovation's marginal cdf has regularly varying tails in the sense of (\ref{one}). Resnick \& St\^aric\^a proved that Hill's estimator is a consistent estimator for $\gamma$ (1995, \cite{Resnick2}) and additionally showed its asymptotic normality (1997, \cite{Resnick4}). In this paper, the aforementioned results shall be extended to the case where one wants to fit a GPD instead of simply estimating the extremal index. In section 2 we are going to fit a special pair of estimators introduced by Zhang (2007, \cite{Zhang}): the likelihood moment estimators (LMEs) and analogously to \cite{Resnick2}, we show in section 3 that the LMEs are consistent estimators for the parameters of the Generalized Pareto distribution. In section 4 we finally discuss further steps needed to show asymptotic normality and why the classical results for tail empirical processes by Resnick \& St\^aric\^a are more promising than the widely used and powerful tools established by Drees (\cite{Drees}) when studying asymptotic behavior of linear processes.
\section{Fitting a Generalized Pareto distribution for a linear process with regularly varying tails}
\setcounter{equation}{0}
Let $\gamma > 0$ and $\{c_j\}_{j \geq 0} \in \mathbb{R}^{\infty}$ and consider a linear process
\begin{equation}
X_n = \sum_{j=0}^{\infty} c_j Z_{n-j} \tag{\text{A.}1} \label{A1}
\end{equation}
whose iid innovations $Z_n$ have a marginal cumulative distribution function (cdf) $G_{Z}$ having regularly varying tails with index $- 1 / \gamma$, i.e.
\begin{equation}
1 - G_Z(z) \sim \pi_1 z^{-1/\gamma}L(z) \hspace{0.4cm} \text{and} \hspace{0.4cm} G_Z(-z)  \sim \pi_2 z^{-1/\gamma}L(z) \hspace{0.4cm} \text{as} \hspace{0.2cm} z \to \infty, \tag{\text{A.}2} \label{A2}
\end{equation}
for $\pi_1, \pi_2 \geq 0$, $\pi_1+ \pi_2 = 1$ and a slowly varying function $L(z)$ satisfying $L(yz) \sim L(z)$ as $z \to \infty$ for any $y > 0$. Clearly, (A.2) implies $1 - G_{|Z|}(z) \sim z^{-1/\gamma}L(z)$, where $G_{|Z|}$ is the (marginal) distribution function of $|Z_n|$.
\vspace{0.25cm}
\newline In order to make some statements about the marginal distribution of $X_n$, we need some restrictions on the sequence of coefficients $\{c_j\}_{j \geq 0}$ in (A.1). We assume that there is at least one $c_j \neq 0$ and that for some $0 < \delta <1/\gamma \wedge 1$:
\begin{equation}
\sum_{j=0}^{\infty} |c_j|^{\delta} < \infty.  \tag{A.3} \label{A3}
\end{equation}
Notice that (A.3) always holds for causal ARMA processes as in (\ref{ARMAintro}) (\cite{Datta}, p. 341). For simplicity and due to their frequent use, (\ref{A1})-(\ref{A3}) will be called the \textbf{Basic Assumptions}.
\vspace{0.25cm}
\newline If we denote the marginal distribution function of $|X_n|$ by $\F_{|X|}$, then, with the help of Lemma 5.2 of \cite{Datta}, we directly conclude that the Basic Assumptions imply
\begin{equation*}
\lim_{t \to \infty} \frac{1-\F_{|X|}(t)}{1 - \G_{|Z|}(t)}= \sum_{k=0}^{\infty} |c_k|^{1/\gamma} := ||c||.
\end{equation*}
Under mild restrictions on the coefficients $c_j$ the tail behavior of $\F_{|X|}$ thus coincides with that of $\G_{|Z|}$ up to the constant $|| c ||$ and consequently, the marginal distribution of the time series $|X_n|$ has also a regularly varying tail with index $- 1 / \gamma$. This fact has an important consequence: By \cite{deHaan}, Theorem 1.1.6, it directly follows that $\F_{|X|}$ belongs to the domain of attraction of the Generalized Extreme Value distribution $\G _\gamma(x) = \exp\left( - (1+\gamma x)^{-1/\gamma} \right)$ and this simply means that there exists a suitable positive function $a_{|X|}(t)$ such that for $t \to \infty$:
\begin{equation}
\frac{b_{|X|}(tx)-b_{|X|}(t)}{a_{|X|}(t)} \to \frac{x^{\gamma} - 1 }{\gamma}, \label{domain3}
\end{equation}
where $b_{|X|}(t) := (1/(1-\F_{|X|}))^{\leftarrow}(t) := \inf\{y : 1/(1-\F_{|X|}(y)) \geq t  \}$ is the $1 - 1/t$-quantile of $\F_{|X|}$. On the other hand, following \cite{Pickands}, Theorem 7, there also exists a positive function $\sigma^*(t)$ such that for $\F_{|X|,t}(x):= \Prob\big{(}|X|-t \leq x \, \big{|} |X| > t\big{)}$:
\vspace{0.2cm}
\begin{equation}
\lim_{t \to \infty} \: \sup_{x > 0} \: \left|\F_{|X|,t}(x) - \left[1- \left(1+ \frac{\gamma}{\sigma^* (t)}x \right)^{-1/\gamma} \right] \right| = 0. \label{domain4}
\end{equation}
\vspace{0.1cm}
\newline
As anticipated in section 1, the function $\Ha_{\gamma, \sigma^*}(x) = 1- \left(1+ \gamma x / \sigma^* (t) \right)^{-1/\gamma}$ is known as the Generalized Pareto distribution with scale function $\sigma^* (t)$. Relation (\ref{domain4}) thus tells us, that the random variable $|X|-t \, \big{|}|X|>t$, the excess above a high threshold $t$, has approximately a Generalized Pareto distribution. In practice, $t$ is replaced by the $(k + 1)$th largest observation $|X|_{n,n-k}$ for a sufficiently large $k << n$. Consequently, the target will be to estimate $\gamma$ and $\sigma^* (t)$ using the $k$ excess-values $|X|_{n,n} - |X|_{n,n-k}, \ldots, |X|_{n,n-k+1} - |X|_{n,n-k}$.
\vspace{0.25cm}
\newline Up to the present day a variety of estimators have been proposed for the two parameters $\gamma$ and $\sigma^*  = \sigma^* (t)$ such as the maximum likelihood estimators (\cite{Smith}), the moment and probability weighted moment estimators (\cite{Hosking}), the likelihood moment estimators (\cite{Zhang}) or the goodness-of-fit estimators (\cite{Huesler}). A nice summary is also available in \cite{deHaan}.
\vspace{0.25cm}
\newline In this paper only the likelihood moment estimators (LMEs) will be studied. The reason behind this choice is that the LMEs can be considered as a generalization of two aforementioned estimators: If we choose $r = -\gamma$, we'll get the usual maximum likelihood equations. By choosing $r = -1$, we just get the goodness-of-fit equations of \cite{Huesler}. As the name suggests, the calculation of the LMEs is based on a mixture between the maximum likelihood and the moment estimation method. In detail, we seek to solve the following system of equations for $\gamma$ and $\sigma^*$:
\vspace{0.2cm}
\begin{eqnarray}
  \frac{1}{k} \sum_{i=0}^{k-1} \log \left(1 + \frac{\gamma}{\sigma^*} \left(|X|_{n,n-i} - |X|_{n,n-k} \right)  \right)  & = & \gamma \label{LME1} \\[0.1cm]
 \frac{1}{k} \sum_{i=0}^{k-1} \left(1 + \frac{\gamma}{\sigma^*} \left(|X|_{n,n-i} - |X|_{n,n-k} \right)  \right)^{r/\gamma} & = & (1-r)^{-1}, \label{LME2}
\end{eqnarray}
\vspace{-0.1cm} \newline
where $r < 1$.
\newline Defining $\gamma / \sigma^*:= \theta^*$ and inserting the first equation into the second one, this task reduces to the one-dimensional problem of solving the following equation for $\theta^*$:
\begin{equation}
\frac{1}{k} \sum_{i=0}^{k-1} \exp\left(\frac{ r \log(1+\theta^* \left(|X|_{n,n-i} - |X|_{n,n-k} \right) )}{  \frac{1}{k} \sum_{i=0}^{k-1} \log \left(1 + \theta^* \left(|X|_{n,n-i} - |X|_{n,n-k} \right)  \right) } \right) - \frac{1 }{(1-r)} = 0. \label{LME3}
\end{equation}
According to \cite{Zhang}, Theorem 2.1., equation (\ref{LME3}) has exactly one solution $\hat{\theta}_{LME}$ for $r < 1/2$ and $ r \neq 0$ which can be easily computed using a Newton-Rhapson algorithm. (Notice that a similar procedure is performed when calculating the maximum likelihood estimators, but the obtained target function may have several roots which are difficult to identify; see \cite{Grimshaw}). $\hat{\gamma}_{LME}$ then is obtained by inserting $\hat{\theta}_{LME}$ in (\ref{LME1}) and from there, it's a straightforward step to derive $\hat{\sigma}_{LME}$.
\section{Consistency}
\setcounter{equation}{0}
In this section, we show that the likelihood moment estimators (the solutions of the system of equations (\ref{LME1}) and (\ref{LME2})) are consistent estimators in case that our Basic Assumptions hold and provided that $r < 0$.
\vspace{0.25cm}
\newline We begin with some preliminaries: Using standard arguments and \cite{deHaan}, Theorem 1.1.6, one can show that the connection between the two ``auxilary'' functions $a_{|X|}$ and $\sigma^*$ from (\ref{domain3}) and (\ref{domain4}) is as follows (cf. \cite{Huesler2}):
\begin{equation}
\sigma(t) := \sigma^*(b_{|X|}(t)) = a_{|X|}( b^{\leftarrow}_{|X|}(b_{|X|}(t))) \sim a_{|X|}(t). \label{aissig}
\end{equation}
Combining (\ref{domain3}) and (\ref{aissig}), we then simply get
\begin{equation*}
\frac{b_{|X|}(tx)-b_{|X|}(t)}{\sigma(t)} \to \frac{x^{\gamma} - 1 }{\gamma},
\end{equation*}
which finally yields (\cite{deHaan}, Lemma 1.2.9)
\begin{equation}
\frac{b_{|X|}(t)}{\sigma(t)} \to \frac{1}{\gamma}, \label{relation}
\end{equation}
both convergences as $t \to \infty$. As we will see, the (new) scaling function $\sigma(t)$ plays a key role throughout this section.
\vspace{0.25cm}
\begin{lemma} Let the Basic Assumptions hold. Also, let's denote the space of right-continuous functions on $[0,\infty)$ by $\operatorname{D}[0,\infty)$. Then for any $z > 0$ we have as $n,k,n/k \to \infty$:
\vspace{0.1cm}
\begin{enumerate}
\item[(a)] for $\mu_X(z,\infty] := \frac{1}{k} \sum_{i=1}^n \mathbf{1} \left\{ |X_i| / b_{|X|}(n/k) \in (z,\infty] \right\}$ and $\mu(z,\infty] := z^{-1/\gamma}$:
\begin{equation}
 \mu_X(z,\infty] \:\: \overset{P}{\to} \: \: \mu(z,\infty], \label{measure}
\end{equation}
\item[(b)] for the $(\lceil ks \rceil + 1)$th largest absolute observation $|X|_{n,n-\lceil ks \rceil}$:
\begin{equation}
 \frac{|X|_{n,n-\lceil ks \rceil}}{b_{|X|}(n/k)} \:\: \overset{P}{\to} \: \: s^{-\gamma} \hspace{0.25cm} \text{in} \: \operatorname{D}[0,\infty), \label{ratio}
\end{equation}
where $\lceil ks \rceil$ is the smallest integer greater than or equal to $ks$ with $s > 0$.
\end{enumerate}
\end{lemma}
\begin{proof} The proof of (b) is a direct consequence of (a) and can simply be looked up in \cite{Resnick3}, p. 81-83. To prove (\ref{measure}), let's rewrite our sequence $X_t$ as
\begin{equation*}
X_t = \sum_{j=0}^{\infty} c_j Z_{t-j} = \sum_{j=0}^{\infty} c_j \sign(Z_{t-j}) |Z_{t-j}| := \sum_{j=0}^{\infty} c'_j |Z_{t-j}|.
\end{equation*}
Furthermore, let $m \in \mathbb{N}\,  \diagdown \{0\}$ and define by $X_t^{(m)} := \sum_{j=0}^{m} c'_j |Z_{t-j}|$ the associated (and truncated) $m$th order moving-average MA($m$). Notice that the $c'_j$s are random now, but as we will see below, one can get around this problem by conditioning on the first $m+1$ innovations $\{\sign(Z_{t-j})\}^m_{j=0}$.
\newline The proof of (\ref{measure}) is now set up in two parts: First we are going to show that for $\mu_X^{(m)}(z,\infty] := \frac{1}{k} \sum_{i=1}^n \mathbf{1} \left\{|X_i^{(m)}|/ b_{|X|}(n/k)\in (z,\infty] \right\}$, the random measure of the truncated time series, we have
\begin{equation}
\mu_X^{(m)}(z,\infty] \overset{P}{\to} \frac{\sum_{j=0}^m |c_j|^{1/\gamma}}{\sum_{j=0}^{\infty} |c_j|^{1/\gamma}} \, \mu(z,\infty] \label{truncatedconv}
\end{equation}
as $n,k,n/k \to \infty$. By showing (\ref{truncatedconv}), the proof of (\ref{measure}) is straightforward then.
\vspace{0.1cm}
\newline Denote the set of all possible (and disjoint) random sequences $\{\sign(Z_{t-j})\}_{j\geq 0}$ by $\mathfrak{S}_\infty$. A single element of this set is further denoted by $S_l^{(\infty)}, l \geq 1$. Define
\begin{equation*}
\tilde{X}_{t,l}^{(m)} := X_t^{(m)}|S_l^{(\infty)} = \sum_{j=0}^{m} c''_{j,l} |Z_{t-j}|,
\end{equation*}
a MA($m$) process with positive iid innovations $|Z_{t-j}|$ consisting of deterministic constants $c''_{j,l}$ that equal either to $c_j$ or $-c_j$, according to $S_l^{(\infty)}$. Finally let
\begin{equation*}
\tilde{X} = \tilde{X}_{t,l} := \lim_{m \to \infty} \tilde{X}_{t,l}^{(m)} =  \sum_{j=0}^{\infty} c''_{j,l} |Z_{t-j}|.
\end{equation*}
If we denote the distribution function of $\tilde{X}$ and $-\tilde{X}$ by $\F_{\tilde{X}}$ and $\F_{-\tilde{X}}$, respectively, we
deduce from \cite{Resnick2}, Proposition 3.2, that for any $m$ and a fixed $l$ we have as $n,k,n/k \to \infty$:
\begin{eqnarray}
\frac{1}{k} \sum_{i=1}^n \mathbf{1} \left\{ \frac{\tilde{X}_{i,l}^{(m)}}{ b_{\tilde{X}}(n/k)} \in (z,\infty] \right\} &\overset{P}{\to}& \frac{ \sum_{j=0}^m  \left([c''_{j,l}]^{1/\gamma} \right)^+} {\sum_{j=0}^{\infty} \left([c''_{j,l}]^{1/\gamma} \right)^+} \, \mu(z,\infty], \label{convtilde1} \\[0.1cm]
\frac{1}{k} \sum_{i=1}^n \mathbf{1} \left\{ \frac{-\tilde{X}_{i,l}^{(m)}}{ b_{-\tilde{X}}(n/k)} \in (z,\infty] \right\} &\overset{P}{\to}& \frac{ \sum_{j=0}^m  \left([c''_{j,l}]^{1/\gamma} \right)^-} {\sum_{j=0}^{\infty} \left([c''_{j,l}]^{1/\gamma} \right)^-} \, \mu(z,\infty], \label{convtilde2}
\end{eqnarray}
where $b_{\tilde{X}}(t)$ and $b_{-\tilde{X}}(t)$ denote the $1-1/t$-quantiles of $\F_{\tilde{X}}$ and $\F_{-\tilde{X}}$, respectively, and $x^+ := \max(x,0)$, $x^- := -\min(x,0)$.
\newline The goal is now to rescale the empirical processes in (\ref{convtilde1}) and (\ref{convtilde2}) in order to obtain new processes that depend on $b_{|X|}(n/k)$. Using Lemma 5.2 of \cite{Datta}, it follows that for $n,k,n/k \to \infty$:
\begin{eqnarray}
\frac{1-\F_{\tilde{X}}(n/k)}{1-\G_{|Z|}(n/k)} &\to& \sum_{j=0}^{\infty} \left([c''_{j,l}]^{1/\gamma} \right)^+, \label{b-ratio1}\\[0.1cm]
\frac{1-\F_{-\tilde{X}}(n/k)}{1-\G_{|Z|}(n/k)} &\to& \sum_{j=0}^{\infty} \left([c''_{j,l}]^{1/\gamma} \right)^-,
\label{b-ratio2}\\[0.1cm]
\frac{1-\F_{|X|}(n/k)}{1-\G_{|Z|}(n/k)} &\to& \sum_{j=0}^{\infty} |c_j|^{1/\gamma}.
\label{b-ratio3}
\end{eqnarray}
Repeating twice the final steps of the proof of Proposition 3.2 of \cite{Resnick2}, (\ref{b-ratio1})-(\ref{b-ratio3}) imply
\begin{equation*}
\frac{b_{|X|}(n/k)}{ b_{\tilde{X}}(n/k)} \to \left( \frac{\sum_{j=0}^{\infty} \left([c''_{j,l}]^{1/\gamma} \right)^+}{\sum_{j=0}^{\infty} |c_j|^{1/\gamma}} \right)^{-\gamma}, \: \:
\frac{b_{|X|}(n/k)}{ b_{-\tilde{X}}(n/k)} \to \left( \frac{\sum_{j=0}^{\infty} \left([c''_{j,l}]^{1/\gamma} \right)^-}{\sum_{j=0}^{\infty} |c_j|^{1/\gamma}} \right)^{-\gamma},
\end{equation*}
both convergences as $n,k,n/k \to \infty$, and from (\ref{convtilde1}) and (\ref{convtilde2}), we finally get
\begin{eqnarray*}
\frac{1}{k} \sum_{i=1}^n \mathbf{1} \left\{ \frac{\tilde{X}_{i,l}^{(m)}}{ b_{|X|}(n/k)} \in (z,\infty] \right\} &\overset{P}{\to}& \frac{ \sum_{j=0}^m  \left([c''_{j,l}]^{1/\gamma} \right)^+} {\sum_{j=0}^{\infty} |c_j|^{1/\gamma}} \, \mu(z,\infty], \\[0.1cm]
\frac{1}{k} \sum_{i=1}^n \mathbf{1} \left\{ \frac{-\tilde{X}_{i,l}^{(m)}}{ b_{|X|}(n/k)} \in (z,\infty] \right\} &\overset{P}{\to}& \frac{ \sum_{j=0}^m  \left([c''_{j,l}]^{1/\gamma} \right)^-} {\sum_{j=0}^{\infty} |c_j|^{1/\gamma}} \, \mu(z,\infty].
\end{eqnarray*}
Hence we conclude that for any $m$ and for arbitrary $l$ as $n,k,n/k \to \infty$:
\begin{eqnarray}
\mu_{\tilde{X},l}^{(m)} (z,\infty] &:=& \frac{1}{k} \sum_{i=1}^n \mathbf{1} \left\{ \frac{|\tilde{X}_{i,l}^{(m)}|}{ b_{|X|}(n/k)} \in (z,\infty] \right\} \nonumber \\[0.1cm] & \overset{P}{\to}& \frac{\sum_{j=0}^{m} |c_j|^{1/\gamma}} {\sum_{j=0}^{\infty} |c_j|^{1/\gamma}} \, \mu(z,\infty]. \label{condmeasureconv}
\end{eqnarray}
The convergence in (\ref{condmeasureconv}) thus shows that the limit of $\mu_{\tilde{X},l}^{(m)} (z,\infty]$ as $n,k,n/k \to \infty$ does not depend on $l$ any more because we only need to consider the absolute value of the $|c_j|$s.
\newline Next consider the set $\mathfrak{S}_m$ of random sequences $\{\sign(Z_{t-j})\}^m_{j = 0}$ with $m \in \mathbb{N}\,  \diagdown \{0\}$. $\mathfrak{S}_m$ consists of $2^{m+1}$ elements and we will denote a single element by $S_{l^*}^{(m)}$, $1 \leq l^* \leq 2^{m+1}$. Clearly, for every $S_{l^*}^{(m)}$ there exists at least one element in $\mathfrak{S}_\infty$, say $S_{l_0}^{(\infty)}$, such that $X_t^{(m)}|S_{l^*}^{(m)} =  X_t^{(m)}|S_{l_0}^{(\infty)} = \tilde{X}_{t,l_0}^{(m)} $ and it follows from (\ref{condmeasureconv}) that for any $l^*$ and any $\varepsilon > 0$:
\begin{eqnarray*}
&& \Prob \left( \left. \left| \mu_{X}^{(m)}(z,\infty] - \frac{\sum_{j=0}^m |c_j|^{1/\gamma}}{\sum_{j=0}^{\infty} |c_j|^{1/\gamma}} \, \mu(z,\infty]\right|  > \varepsilon \hspace{0.25cm} \right|  \: S_{l^*}^{(m)}  \right) \\[0.1cm]
&=& \Prob \left( \left|  \frac{1}{k} \sum_{i=1}^n \mathbf{1} \left\{ \frac{|X_i^{(m)} \, | \, S_{l^*}^{(m)}| }{ b_{|X|}(n/k)} \in (z,\infty] \right\} - \frac{\sum_{j=0}^m |c_j|^{1/\gamma}}{\sum_{j=0}^{\infty} |c_j|^{1/\gamma}} \, \mu(z,\infty]\right|  > \varepsilon \right) \\[0.1cm]
&=& \Prob \left( \left|  \frac{1}{k} \sum_{i=1}^n \mathbf{1} \left\{ \frac{|X_i^{(m)} \, | \, S_{l_0}^{(\infty)}| }{ b_{|X|}(n/k)} \in (z,\infty] \right\} - \frac{\sum_{j=0}^m |c_j|^{1/\gamma}}{\sum_{j=0}^{\infty} |c_j|^{1/\gamma}} \, \mu(z,\infty]\right|  > \varepsilon \right) \\[0.1cm]
&=&  \Prob \left( \left| \mu_{\tilde{X},l_0}^{(m)}(z,\infty] - \frac{\sum_{j=0}^m |c_j|^{1/\gamma}}{\sum_{j=0}^{\infty} |c_j|^{1/\gamma}} \, \mu(z,\infty]\right|  > \varepsilon \right) \to 0,
\end{eqnarray*}
$n,k,n/k \to \infty$. By an application of the formula of total probability it then follows that for any $m \in \mathbb{N}\,  \diagdown \{0\}$:
\begin{eqnarray*}
&&\Prob \left(\left| \mu_X^{(m)}(z,\infty] - \frac{\sum_{j=0}^m |c_j|^{1/\gamma}}{\sum_{j=0}^{\infty} |c_j|^{1/\gamma}} \, \mu(z,\infty]\right| > \varepsilon \right) \\[0.1cm]
&=& \sum_{l^*=1}^{2^{m+1}} \Prob \left( \left. \left| \mu_X^{(m)}(z,\infty] - \frac{\sum_{j=0}^m |c_j|^{1/\gamma}}{\sum_{j=0}^{\infty} |c_j|^{1/\gamma}} \, \mu(z,\infty]\right|  > \varepsilon \hspace{0.25cm} \right| \: S_{l^*}^{(m)} \right) \Prob\left(S_{l^*}^{(m)}\right) \\[0.1cm]
&\leq&  \sup_{l^*}  \Prob \left( \left. \left| \mu_X^{(m)}(z,\infty] - \frac{\sum_{j=0}^m |c_j|^{1/\gamma}}{\sum_{j=0}^{\infty} |c_j|^{1/\gamma}} \, \mu(z,\infty]\right|  > \varepsilon \hspace{0.25cm} \right| \: S_{l^*}^{(m)} \right)  \sum_{l^*=1}^{2^{m+1}} \Prob\left(S_{l^*}^{(m)}\right) \\[0.25cm] &\to& 0,
\end{eqnarray*}
$n,k,n/k \to \infty$. This proves (\ref{truncatedconv}) and finishes the first part of the proof.
\newline To show that the convergence in (\ref{measure}) holds, it now suffices to check whether
\begin{equation*}
\lim_{m \to \infty} \: \limsup_{n,k,n/k \to \infty} \Prob\left( \left| \mu_X^{(m)}(z,\infty] - \mu_X(z,\infty] \right| > \varepsilon \right) = 0.
\end{equation*}
(\cite{Resnick3}, Theorem 3.5). Based on the proof of \cite{Resnick2}, Proposition 3.3, define the events
\begin{eqnarray*}
  A_i &:=& \left\{ 0 \leq \frac{|X_i^{(m)}|}{ b_{|X|}(n/k)} - \frac{|X_i|}{ b_{|X|}(n/k)} \leq \delta \right\}, \\[0.1cm]
  B_i &:=& \left\{ 0 \leq  \frac{|X_i|}{ b_{|X|}(n/k)} - \frac{|X_i^{(m)}|}{ b_{|X|}(n/k)} \leq \delta\right\}, \\[0.1cm]
  C_i &:=& \left\{ \left|\frac{|X_i^{(m)}|}{ b_{|X|}(n/k)} - \frac{|X_i|}{ b_{|X|}(n/k)} \right| > \delta \right\}
\end{eqnarray*}
for some $0 < \delta < z$. Then for $m \in \mathbb{N}\, \diagdown \{0\}$ and $\varepsilon > 0$:
\begin{eqnarray*}
&& \Prob\left( \left| \mu_X^{(m)}(z,\infty] - \mu_X(z,\infty] \right| > \varepsilon \right) \\[0.25cm]
&\leq& \Prob\left( \frac{1}{k} \sum_{i=1}^n \left| \mathbf{1} \left\{ \frac{|X_i^{(m)}|}{ b_{|X|}(n/k)} \in (z,\infty] \right\} - \mathbf{1} \left\{ \frac{|X_i|}{ b_{|X|}(n/k)} \in (z,\infty] \right\} \right| \mathbf{1}_{A_i} > \varepsilon/3 \right) \\[0.1cm]
&& + \Prob\left( \frac{1}{k} \sum_{i=1}^n \left| \mathbf{1} \left\{ \frac{|X_i^{(m)}|}{ b_{|X|}(n/k)} \in (z,\infty] \right\} - \mathbf{1} \left\{ \frac{|X_i|}{ b_{|X|}(n/k)} \in (z,\infty] \right\} \right| \mathbf{1}_{B_i} > \varepsilon/3 \right) \\[0.1cm]
&& + \Prob\left( \frac{1}{k} \sum_{i=1}^n \mathbf{1}_{C_i} > \varepsilon/3 \right)\\[0.1cm]
&\leq& \Prob\left( \frac{1}{k} \sum_{i=1}^n \left[ \mathbf{1} \left\{ \frac{|X_i^{(m)}|}{ b_{|X|}(n/k)} \in (z,\infty] \right\} - \mathbf{1} \left\{ \frac{|X_i|}{ b_{|X|}(n/k)} \in (z,\infty] \right\} \right] \mathbf{1}_{A_i} > \varepsilon/3 \right) \\[0.1cm]
&& + \Prob\left( \frac{1}{k} \sum_{i=1}^n \left[ \mathbf{1} \left\{ \frac{|X_i|}{ b_{|X|}(n/k)} \in (z,\infty] \right\} - \mathbf{1} \left\{ \frac{|X_i^{(m)}|}{ b_{|X|}(n/k)} \in (z,\infty] \right\} \right] \mathbf{1}_{B_i} > \varepsilon/3 \right) \\[0.1cm]
&& + \frac{3}{\varepsilon} \frac{n}{k} \E \left(  \mathbf{1}\left\{ \left| \frac{|X_1^{(m)}|}{ b_{|X|}(n/k)} -  \frac{|X_1|}{ b_{|X|}(n/k)} \right| > \delta \right\} \right) \\[0.25cm]
&:=& \text{I}_1 + \text{I}_2 + \text{I}_3.
\end{eqnarray*}
Here Chebyshev's inequality was used in the second last line. Now
\begin{equation*}
\text{I}_1 \leq \Prob\left( \frac{1}{k} \sum_{i=1}^n \left[ \mathbf{1} \left\{\frac{|X_i^{(m)}|}{ b_{|X|}(n/k)} \in (z,\infty] \right\} - \mathbf{1}  \left\{\frac{|X_i^{(m)}|}{ b_{|X|}(n/k)} \in (z + \delta,\infty] \right\} \right] > \varepsilon/3 \right) \\[0.1cm]
\end{equation*}
and since we know from (\ref{truncatedconv}) that
\begin{eqnarray*}
&& \frac{1}{k} \sum_{i=1}^n \mathbf{1} \left\{\frac{|X_i^{(m)}|}{ b_{|X|}(n/k)} \in (z,\infty] \right\} - \frac{1}{k} \sum_{i=1}^n \mathbf{1}  \left\{\frac{|X_i^{(m)}|}{ b_{|X|}(n/k)} \in (z + \delta,\infty] \right\} \\[0.1cm]
&&\overset{P}{\to} \frac{\sum_{j=0}^m |c_j|^{1/\gamma}}{\sum_{j=0}^{\infty} |c_j|^{1/\gamma}}\left(z^{-1/\gamma} - (z+\delta)^{-1/\gamma} \right)
\end{eqnarray*}
we conclude $ \text{I}_1  \to 0$ as $n,k,n/k \to \infty$ letting $\delta := \delta(\varepsilon) \searrow 0.$
\newline Similarly
\begin{equation*}
\text{I}_2 \leq \Prob\left( \frac{1}{k} \sum_{i=1}^n \left[ \mathbf{1} \left\{\frac{|X_i^{(m)}|}{ b_{|X|}(n/k)} \in (z - \delta,\infty] \right\} - \mathbf{1}  \left\{\frac{|X_i^{(m)}|}{ b_{|X|}(n/k)} \in (z,\infty] \right\} \right] > \varepsilon/3 \right), \\[0.1cm]
\end{equation*}
and by another application of (\ref{truncatedconv}) we conclude $ \text{I}_2  \to 0$ as $n,k,n/k \to \infty$ letting again $\delta := \delta(\varepsilon) \searrow 0$.
Finally, for $\delta' > 0$ and $n, k, n/k$ large:
\begin{eqnarray*}
|\text{I}_3| &=& \frac{3}{\varepsilon} \frac{n}{k} \Prob \left( \left| |X_1^{(m)}| -  |X_1| \right|  > \delta \, b_{|X|}(n/k) \right) \leq \frac{3}{\varepsilon}  \frac{n}{k} \Prob \left( \left| \sum_{m+1}^{\infty} c_j Z_{1-j} \right|  > \delta \, b_{|X|}(n/k) \right) \\[0.1cm]
&\sim& \frac{3}{\varepsilon}  \frac{\Prob \left( \left| \sum_{m+1}^{\infty} c_j Z_{1-j} \right|  > \delta \, b_{|X|}(n/k) \right)}{\Prob \left( |Z_{1}|  >  b_{|Z|}(n/k) \right)} \leq \frac{3}{\varepsilon}  \frac{\Prob \left( \left| \sum_{m+1}^{\infty} c_j Z_{1-j} \right|  > \delta' \, b_{|Z|}(n/k) \right)}{\Prob \left( |Z_{1}|  >  b_{|Z|}(n/k) \right)} \\[0.1cm]
&\sim& \frac{3}{\varepsilon} \sum_{m+1}^{\infty} |c_j/\delta'|^{1/\gamma}.
\end{eqnarray*}
as $n,k,n/k \to \infty$, where we used the fact that $||a-b| - |a|| \leq |b|$ for $ a, b \in \mathbb{R}$ in the first line, $k/n \sim \Prob(|Z_1| > b_{|Z|}(n/k))$ (Theorem 3.6 of \cite{Resnick3}) and $b_{|X|}(n/k) \sim \text{const.} \, b_{|Z|}(n/k)$ (\cite{Datta}, Lemma 5.2, and \cite{Resnick1}, Proposition 0.8(vi)) in line 2, and again \cite{Datta}, Lemma 5.2, in line 3. Since the last sum converges to zero as $m \to \infty$, the proof of (\ref{measure}) is complete. \qed
\end{proof}
We next study the behavior of the tail empirical measure of $|X_i|-|X|_{n,n-k}$, $1 \leq i \leq n$.
\begin{lemma} Let the Basic Assumptions hold and let $\sigma(t)$ be defined as in (\ref{aissig}), then, as $n,k,n/k \to \infty$, we have for $z \geq 0$:
\begin{equation}
\frac{1}{k} \sum_{i=1}^n \mathbf{1} \left\{ |X_i|-|X|_{n,n-k} > \sigma(n/k) z \right\} \: \: \overset{P}{\to} \: \: (z\gamma + 1)^{-1/\gamma}, \label{2.1.2}
\end{equation}
\end{lemma}
\begin{proof} If $z = 0$, the proof is trivial and hence we lay our focus on the case where $z > 0$. By (\ref{relation}) and (\ref{ratio}), we have that
\begin{equation*}
\frac{\sigma(n/k)}{ b_{|X|}(n/k)} +  \frac{|X|_{n,n-k}}{z \, b_{|X|}(n/k)} \:\: \overset{P}{\to} \: \: \gamma + \frac{1}{z},
\end{equation*}
$n,k,n/k \to \infty$, and together with (\ref{measure}), we get the following joint convergence in probability:
\begin{equation}
\left(\mu_X(z,\infty],\frac{\sigma(n/k)}{ b_{X}(n/k)} +  \frac{|X|_{n,n-k}}{z \, b_{|X|}(n/k)}\right) \: \: \overset{P}{\to} \: \:  \left(\mu(z,\infty], \gamma + \frac{1}{z} \right). \label{jointconv}
\end{equation}
(\cite{Billingsley}, Theorem 4.4).
\newline The final step then is to make use of the operator
\begin{equation*}
T: \mathbb{R}^+ \times \mathbb{R}^+  \to \mathbb{R}^+ \, \text{,} \: \:  T(\mu_X(z,\infty],x) \mapsto \mu_X(zx,\infty].
\end{equation*}
Since $T$ is continuous in $(\mu(z,\infty],x)$ (see \cite{Resnick2}, proof of Proposition 2.1), a simple application of the continuous mapping theorem on (\ref{jointconv}) yields
\begin{align*}
& T\left(\mu_X(z,\infty],\frac{\sigma(n/k)}{ b_{X}(n/k)} +  \frac{|X|_{n,n-k}}{z \, b_{|X|}(n/k)}\right) \\[0.1cm] &=
\frac{1}{k} \sum_{i=1}^n \mathbf{1} \left\{ \frac{|X_i|}{b_{|X|}(n/k)} > \frac{ \sigma(n/k) z + |X|_{n,n-k}}{b_{|X|}(n/k)} \right\} \\[0.1cm]
&= \frac{1}{k} \sum_{i=1}^n \mathbf{1} \left\{ |X_i|-|X|_{n,n-k} > \sigma(n/k) z \right\} \\[0.1cm]
&\overset{P}{\to} T\left(\mu(z,\infty], \gamma + \frac{1}{z} \right) = \mu(z\gamma + 1,\infty].
\end{align*}
\qed
\end{proof}
The next result gives the link between the tail empirical measure and a generalized form of the two tail array sums involved in the likelihood moment equations.
\begin{lemma} Assume that the Basic Assumptions hold and let $\sigma(t)$ be defined as in (\ref{aissig}). Then, as $n,k,n/k \to \infty$, we have for $r < 0$ and $x,y > 0$:
\begin{align}
& \frac{1}{k} \sum_{i=0}^{k-1} \log \left(1 + \frac{\gamma x}{\sigma(n/k)} \left(|X|_{n,n-i} - |X|_{n,n-k} \right)  \right) \nonumber \\[0.1cm] &  \overset{P}{\to} \int_0^{\infty} \frac{dz}{\left(\frac{z}{x} + 1 \right)^{1/\gamma} (z+1) }:= \psi_1(x), \label{2.1.3.1} \\[0.1cm]
& \frac{1}{k} \sum_{i=0}^{k-1} \left(1 + \frac{\gamma x}{\sigma(n/k)} \left(|X|_{n,n-i} - |X|_{n,n-k} \right)  \right)^{\hspace{-0.05cm} r/y}  \nonumber \\[0.1cm] &  \overset{P}{\to}  \frac{r}{y} \int_0^{\infty} \frac{dz}{\left(\frac{z}{x} + 1 \right)^{1/\gamma} (z+1)^{1-r/y}} + 1:= \psi_2(x,y). \label{2.1.3.2}
\end{align}
\end{lemma}
\begin{proof} We are going to prove (\ref{2.1.3.2}). (\ref{2.1.3.1}) can be shown using the very same steps.
\newline By some straightforward steps and denoting $X^+ := \max(X,0)$, it is easy to see that
\begin{eqnarray*}
&& \frac{1}{k} \sum_{i=0}^{k-1} \left(1 + \frac{\gamma x}{\sigma(n/k)} \left(|X|_{n,n-i} - |X|_{n,n-k} \right)  \right)^{r/y} \\[0.1cm] &=& \frac{1}{k} \sum_{i=1}^{n} \left(1 + \frac{\gamma x}{\sigma(n/k)} \left(|X_i| - |X|_{n,n-k} \right)^+  \right)^{r/y} - \frac{n-k}{k}\\[0.1cm] &=& \frac{1}{k} \sum_{i=1}^{n} \frac{r}{y} \int_0^{\gamma x \left(|X_i| - |X|_{n,n-k} \right)^+ / \sigma(n/k) }\frac{dz}{(z+1)^{1 - r/y}} + 1 \\[0.1cm] &=& \frac{r}{y} \underbrace{\int_0^\infty \frac{1}{k} \sum_{i=1}^n \mathbf{1} \left\{ |X_i|-|X|_{n,n-k} > \frac{\sigma(n/k) z}{ \gamma x} \right\} \frac{dz}{(z+1)^{1 - r/y}}}_{:= \displaystyle \chi_{n,x,y}} + 1.
\end{eqnarray*}
Next define $\chi_{n,x,y}(t) := \int_0^{t}  \frac{1}{k} \sum_{i=1}^n \mathbf{1} \left\{ \frac{\gamma x }{\sigma(n/k)}(|X_i|-|X|_{n,n-k}) > z \right\} \, \frac{dz}{(z+1)^{1 - r/y}}$ for $t < \infty$. Notice that the mapping $f(\cdot) \mapsto \int_0^{t} f(s) /(s+1)^{1 - r/y} \, ds$ is continuous (we are integrating over a finite region) and it follows from (\ref{2.1.2}) and the continuous mapping theorem that for all $t$:
\begin{equation*}
\chi_{n,x,y}(t)  \: \: \overset{P}{\to} \: \:  \int_0^{t} \frac{dz}{\left(\frac{z}{x} + 1 \right)^{1/\gamma} (z+1)^{1-r/y}}.
\end{equation*}
The result thus follows again if we are able to show, that for any $\varepsilon > 0$:
\begin{equation*}
\lim_{t \to \infty} \: \limsup_{n,k,n/k \to \infty} \Prob(|\chi_{n,x,y} - \chi_{n,x,y}(t)| > \varepsilon) = 0.
\end{equation*}
Pick $\delta > 0, \, \varepsilon' < 1$ and $ 0 < \varepsilon'' < 1/\gamma - r/y$. Then there exists a pair $(n_0, k_0) = (n_0(\delta, \varepsilon', \varepsilon''),  k_0(\delta, \varepsilon',\varepsilon''))$ such that for all $n > n_0$ and $k > k_0$:
\begin{eqnarray*}
&& \Prob(|\chi_{n,x,y} - \chi_{n,x,y}(t)| > \varepsilon)
\\[0.1cm] &=& \Prob \left( \int_t^\infty \mu_X \left( \frac{z \sigma(n/k)}{\gamma x b_{|X|}(n/k)} +  \frac{|X|_{n,n-k}}{ b_{|X|}(n/k)} , \infty \right] \frac{dz}{(z + 1)^{1 - r/y}} > \varepsilon \right)
\\[0.1cm]  &\overset{(\ref{relation})}{\leq}& \Prob \left( \int_t^\infty \mu_X \left( z ( 1- \varepsilon')/x + (1-\delta), \infty \right] \frac{dz}{(z + 1)^{1 - r/y}} > \varepsilon \right)
\\[0.1cm] && + \Prob \left(  \frac{|X|_{n,n-k}}{ b_{|X|}(n/k)} < 1 - \delta \right)
\\[0.1cm] &\overset{(\ref{ratio})}{\leq}& \Prob \left( \int_t^\infty \mu_X \left( (1 - \delta - \varepsilon')(z/x+1) , \infty \right] \frac{dz}{(z + 1)^{1 - r/y}} > \varepsilon \right) + o(1)
\\[0.1cm] &\leq&  \varepsilon^{-1} \int_t^\infty \frac{n}{k} \Prob \left(\frac{|X_1|}{ b_{|X|}(n/k)} > \frac{(1 - \delta - \varepsilon')}{x} \, (z+x) \right) \frac{dz}{(z + 1)^{1 - r/y}}  + o(1)
\\[0.25cm] &\leq& \text{const.} \int_t^\infty \frac{dz}{(z + x)^{1/\gamma - \varepsilon''}\cdot (z+1)^{1-r/y}}  + o(1)
\\[0.25cm] &\sim& \text{const.} \cdot t^{\, r/y - 1/\gamma + \varepsilon''} + o(1),
\end{eqnarray*}
where we used Chebyshev's inequality in line 6 and Potter's inequality (\cite{Bingham}) in line 7. Hence for $n, k, n/k$ large:
\begin{align*}
& \lim_{t \to \infty} \: \limsup_{n,k,n/k \to \infty} \Prob(|\chi_{n,x,y} - \chi_{n,x,y}(t)| > \varepsilon) \\[0.1cm] \leq& \lim_{t \to \infty} \: \limsup_{n,k,n/k \to \infty}  \text{const.} \cdot t^{\, r/y - 1/\gamma + \varepsilon''} + o(1) = 0.
\end{align*}\qed
\end{proof}
Notice that for the special cases where $x=1$ and $y = \gamma$, the integrals in (\ref{2.1.3.1}) and (\ref{2.1.3.2}) are easy to calculate so that $\psi_1(1) = \gamma$ and $\psi_2(1,\gamma) = \psi_2(1,\psi_1(1)) = 1/(1-r).$ This will be essential in our main theorem which follows next.
\begin{theorem} Assume that the Basic Assumptions hold and let $\sigma(t)$ be defined as in (\ref{aissig}). Then for $r < 0$, we have as $n,k,n/k \to \infty$:
\begin{equation}
  \left( \begin{array}{cc}
   \hat{\gamma}_{LME} \\[0.1cm]
   \hat{\sigma}_{LME} / \sigma(n/k)
  \end{array} \right)
   \: \: \overset{P}{\to} \: \:
     \left( \begin{array}{cc}
  \gamma \\[0.1cm]
   1
  \end{array} \right). \label{2.20}
\end{equation}
\end{theorem}
\begin{proof}
We are going to show that $\hat{\theta}_{LME}/\theta(n/k) \overset{P}{\to} 1$ as $n,k,n/k \to \infty$, where $\hat{\theta}_{LME}$ is the (unique) solution of equation (\ref{LME3}) and $\theta(n/k) := \gamma/\sigma(n/k)$. If this convergence holds, then $\hat{\gamma} = \frac{1}{k} \sum_{i=0}^{k-1} \log (1 +\hat{\theta}_{LME} \left(|X|_{n,n-i} - |X|_{n,n-k} \right) ) \overset{P}{\to} \psi_1(1) = \gamma$ by (\ref{2.1.3.1}) and the continuous mapping theorem and hence we have that $\hat{\sigma}_{LME}/\sigma(n/k) = \hat{\gamma}_{LME}/\gamma \cdot \theta(n/k) / \hat{\theta}_{LME} \overset{P}{\to} 1 $ as $n,k,n/k \to \infty$. The joint convergence in (\ref{2.20}) finally follows by \cite{Billingsley}, Theorem 4.4.
\vspace{0.25cm}
\newline Define $Y_i := |X|_{n,n-i} - |X|_{n,n-k}$, $\boldsymbol{Y}:= \{Y_0, Y_1, \ldots, Y_{k-1} \}$ and for $x,y > 0$: $\psi_1(x,\boldsymbol{Y}) := (1/k) \sum^{k-1}_{i=0} \log(1+x  \theta(n/k)  Y_i)$ and $\psi_2(x,y,\boldsymbol{Y}) := (1/k) \sum^{k-1}_{i=0} (1+x  \theta(n/k)  Y_i)^{r/y}$.
\vspace{0.25cm}
\newline Now, from Result 3.3 and again Theorem 4.4 of \cite{Billingsley}, it easily follows as $n,k,n/k$ $\to \infty$ that
\begin{equation*}
(\psi_1(x,\boldsymbol{Y}), \psi_2(x,y, \boldsymbol{Y})) \overset{P}{\to} (\psi_1(x), \psi_2(x,y)).
\end{equation*}
Next consider the operator
\begin{equation*}
\tilde{T}: \mathbb{R}^+ \times \mathbb{R}^+  \to \mathbb{R}^+ \, \text{,} \: \:  \tilde{T}(\psi_1(x,\boldsymbol{Y}),\psi_2(x,y, \boldsymbol{Y})) \mapsto \psi_2(x,\psi_1(x,\boldsymbol{Y}), \boldsymbol{Y}).
\end{equation*}
Since both $\psi_1(x,\boldsymbol{Y})$ and $\psi_2(x,y, \boldsymbol{Y})$ are continuous in $x,y >0$, $\tilde{T}$ is also continuous and hence by the continuous mapping theorem:
\begin{equation*}
 \psi_2(x,\psi_1(x,\boldsymbol{Y}), \boldsymbol{Y}) \overset{P}{\to} \psi_2(x,\psi_1(x)),
\end{equation*}
which clearly yields
\begin{equation}
\psi(x,\boldsymbol{Y}) := \psi_2(x,\psi_1(x,\boldsymbol{Y}), \boldsymbol{Y}) - \frac{1}{(1-r)}\overset{P}{\to} \psi_2(x,\psi_1(x)) - \frac{1}{(1-r)} := \psi(x) \label{psi}
\end{equation}
as $n,k,n/k \to \infty$. Notice thereby that $\psi(x,\boldsymbol{Y})$ is a strictly decreasing function in $x \in (0,\infty)$ (cf. \cite{Zhang}, proof of Theorem 2.1) and according to Lemma 4, its limit function $\psi(x)$ has the very same property.
\vspace{0.25cm}
\newline Recall from (\ref{relation}) and (\ref{ratio}) that $\theta(n/k) Y_0 = \gamma Y_0/\sigma(n/k) \sim  Y_0/ b_{|X|}(n/k) \overset{P}{\to} \infty$ as $n,k,n/k \to \infty$. Hence there exists no solution $\hat{\theta}_{LME}$ of (\ref{LME3}) such that  $\hat{\theta}_{LME}/\theta(n/k)$ converges to a negative value in probability.
\vspace{0.25cm}
\newline Let's consequently assume first that $\hat{\theta}_{LME}/\theta(n/k) \overset{P}{\to} a \in (0,\infty)$. Obviously, the mapping $\hat{\theta}_{LME}/\theta(n/k) $ $ \mapsto \psi(\hat{\theta}_{LME}/\theta(n/k),\boldsymbol{Y})$ is continuous since $\psi(x,\boldsymbol{Y})$ is continuous in $x \in (0,\infty)$ and from (\ref{psi}), it follows once more by \cite{Billingsley}, Theorem 4.4, and the continuous mapping theorem that
\begin{equation}
 \psi(\hat{\theta}_{LME}/\theta(n/k),\boldsymbol{Y}) - \psi(a) \overset{P}{\to} 0, \label{psylim}
\end{equation}
$n,k,n/k \to \infty$. But by definition of $\hat{\theta}_{LME}$ (see (\ref{LME3})), $\psi(\hat{\theta}_{LME}/\theta(n/k),\boldsymbol{Y}) = 0$ and thus the convergence in (\ref{psylim}) holds if and only if $a = 1$ because only in that case $\psi(a) = 0$ by the monotonicity of $\psi$.
\vspace{0.25cm}
\newline Thus, the proof is complete if we are able to show that neither $\hat{\theta}_{LME}/\theta(n/k) \overset{P}{\to} 0$ nor $\hat{\theta}_{LME}/\theta(n/k) \overset{P}{\to} \infty$.
\newline For the first case, use again the monotonicity of $\psi(x ,\boldsymbol{Y})$ to see that for any $\varepsilon \in (0,1)$:
\begin{eqnarray*}
&& \lim_{n,k,n/k \to \infty} \Prob\left(  \left| \hat{\theta}_{LME}/\theta(n/k)  \right| \geq  \varepsilon \right) = \lim_{n,k,n/k \to \infty} \Prob\left( \hat{\theta}_{LME}/\theta(n/k)   \geq \varepsilon \right)  \\[0.1cm]
&=& \lim_{n,k,n/k \to \infty} \Prob\left( \psi(\hat{\theta}_{LME}/\theta(n/k),\boldsymbol{Y})  \leq \psi(\varepsilon
,\boldsymbol{Y})\right) = \lim_{n,k,n/k \to \infty} \Prob\left( 0  \leq \psi(\varepsilon ,\boldsymbol{Y})\right)  \\[0.25cm]
&=& 1,
\end{eqnarray*}
because $ \psi(\varepsilon ,\boldsymbol{Y}) \overset{P}{\to} \psi(\varepsilon) > 0$. Hence
\begin{equation*}
\Prob\left(  \left| \hat{\theta}_{LME}/\theta(n/k)  \right| <  \varepsilon \right) \to 0
\end{equation*} as $n,k,n/k \to \infty$.
\newline Similarly for the second case:
\begin{eqnarray*}
&& \lim_{n,k,n/k \to \infty} \Prob\left(  \left| \hat{\theta}_{LME}/\theta(n/k)  \right| >  \varepsilon^{-1} \right) = \lim_{n,k,n/k \to \infty} \Prob\left( 0  < \psi(\varepsilon^{-1} ,\boldsymbol{Y})\right) = 0,
\end{eqnarray*}
because $ \psi(\varepsilon^{-1} ,\boldsymbol{Y}) \overset{P}{\to} \psi(\varepsilon^{-1}) < 0$ by the monotonicity of $\psi(x ,\boldsymbol{Y})$. This concludes the proof. \qed
\end{proof}
\begin{lemma}
$\psi(x)$ defined in (\ref{psi}) is a strictly decreasing function in $x \in (0,\infty)$.
\end{lemma}
\begin{proof}
Let $\psi(x,\mathbf{Y})$ be defined as in (\ref{psi}). Then the aim is to show that $\psi'(x,\mathbf{Y}) := d\, \psi(x,\mathbf{Y}) / dx$ converges in probability to a negative limit for all $x > 0$ as $n,k,n/k \to \infty$.
\vspace{0.25cm}
\newline Recall that $Y_i = |X|_{n,n-i}-|X|_{n,n-k}$ and denote $Z_{i,x} := \log(1 + x\,\theta(n/k)Y_i)$ and $\bar{Z}_x := \frac{1}{k} \sum_{j=0}^{k-1} Z_{j,x}$ (which in fact is equivalent to $\psi_1(x,\mathbf{Y})$ in (\ref{2.1.3.1})) for $i = 0, 1, \ldots, k-1$ and $x > 0$. Notice thereby that $Z_{i,x} > 0$ for all $i$ and $x > 0$ and furthermore $Z_{i,x} > Z_{j,x}$ if $i < j$. Also, $\bar{Z}_x \overset{P}{\to} \psi_1(x)$ as $n,k,n/k \to \infty$ by (\ref{2.1.3.1}).
\newline Further recall by (\ref{relation}) and (\ref{ratio}) that for $l \in \{ 1, 2, 3, 4 \}$:
\begin{equation*}
\theta(n/k)Y_{\lceil k\cdot l / 5 \rceil } \sim  Y_{\lceil k\cdot l / 5 \rceil }/ b_{|X|}(n/k)  \overset{P}{\to} \left( \frac{l}{5}\right)^{-\gamma}  \hspace{-0.10cm} - 1 ,
\end{equation*}
and thus, by a simple application of the continuous mapping theorem,
\begin{equation}
Z_{\lceil k\cdot l / 5 \rceil , x}  \overset{P}{\to} \log \left( 1 + x  \left( \left( \frac{l}{5}\right)^{-\gamma}  \hspace{-0.10cm} - 1 \right) \right) \label{zconv}
\end{equation}
as $n,k,n/k \to \infty.$
\newline Now, going along the steps of the proof of Theorem 2.1 in \cite{Zhang}, $\psi'(x,\mathbf{Y})$ may be expressed in the following way:
\begin{equation*}
\psi'(x,\mathbf{Y}) = \bar{Z}_x^{-2} k^{-2} \sum_{0 \leq i < j \leq k-1} Z_{i,x}Z_{j,x} \cdot (u(Z_{i,x}) - u(Z_{j,x})) \cdot (v(Z_{i,x}) - v(Z_{j,x})),
\end{equation*}
where $u(a) := r \exp(ra/\bar{Z}_x)/x$ is strictly increasing and $v(a) := (1-\exp(a))/a$ is strictly decreasing in $a$. Hence $\psi'(x,\mathbf{Y}) < 0$ for all $x > 0$. If we define the set $A_{i,j}^* := \{\lceil k/5 \rceil \leq i \leq \lceil 2 k/5 \rceil, \lceil 3k/5 \rceil \leq j \leq \lceil 4 k/5 \rceil \}$, then, by straightforward steps:
\begin{eqnarray*}
&& \psi'(x,\mathbf{Y}) \\[0.25cm]
&<&  k^{-2} \sum_{A_{i,j}^*} Z_{i,x}Z_{j,x}/\bar{Z}_x^2 \cdot (u(Z_{i,x}) - u(Z_{j,x})) \cdot (v(Z_{i,x}) - v(Z_{j,x})) \\[0.1cm]
&\leq&  k^{-2} \sum_{A_{i,j}^*}\max_{A_{i,j}^*} \left\{ Z_{i,x}Z_{j,x}/\bar{Z}_x^2 \cdot (u(Z_{i,x}) - u(Z_{j,x})) \cdot (v(Z_{i,x}) - v(Z_{j,x})) \right\} \\[0.1cm]
&\sim&  25^{-1} \max_{A_{i,j}^*} \left\{ Z_{i,x}Z_{j,x}/\bar{Z}_x^2 \cdot (u(Z_{i,x}) - u(Z_{j,x})) \cdot (v(Z_{i,x}) - v(Z_{j,x})) \right\} \\[0.25cm]
&\leq&  25^{-1} \min_{A_{i,j}^*} \left\{ Z_{i,x}Z_{j,x}/\bar{Z}_x^2 \right\} \min_{A_{i,j}^*} \left\{ u(Z_{i,x}) - u(Z_{j,x}) \right\}  \max_{A_{i,j}^*} \left\{ v(Z_{i,x}) - v(Z_{j,x}) \right\} \\[0.25cm]
&=&  25^{-1}  \left\{ Z_{\lceil 2k/5 \rceil ,x}Z_{\lceil 4k/5 \rceil ,x}/\bar{Z}_x^2 \right\} \left\{ u(Z_{\lceil 2k/5 \rceil ,x}) - u(Z_{\lceil 3k/5 \rceil ,x})\right\}  \\[0.25cm] &&  \cdot  \left\{ v(Z_{\lceil 2k/5 \rceil ,x}) - v(Z_{\lceil 3k/5 \rceil ,x})\right\} := 25^{-1} (\text{I}_1 \cdot  \text{I}_2 \cdot  \text{I}_3 ).
\end{eqnarray*}
Using the fact that both $u(a), v(a)$ are continuous in $a > 0$, (\ref{2.1.3.1}), (\ref{zconv}) and several applications of the continuous mapping theorem simply yield $\text{I}_1 \overset{P}{\to} \text{const.} > 0$, $\text{I}_2 \overset{P}{\to} \text{const.} > 0$ and $\text{I}_3 \overset{P}{\to} \text{const.} < 0$ so that $\text{I}_1 \cdot \text{I}_2 \cdot \text{I}_3 \overset{P}{\to} \text{const.} < 0$ as $n,k,n/k \to \infty$. \qed
\end{proof}
\section{Concluding remarks and outlook}
At first glance it may seem a bit unconventional to use the classical results of Resnick \& St\^aric\^a rather than a well-known set of powerful tools established by Drees (\cite{Drees}) when studying asymptotic behavior of linear processes. To our misfortune, it is a very tough question whether Drees' main Theorem 2.1. can be applied to MA($\infty$)-processes as the author himself admits that one of its conditions ``is more complicated to check [...] for general linear time series'' (p. 635). Nevertheless, he was able to show that all conditions hold for an AR($1$)-process (section 3.2), a result which was recently extended to arbitrary AR($p$)-processes, $p \in \mathbb{N}\,  \diagdown \{0\}$, by Kulik et al. (\cite{Kulik}). It thus would be of great interest to extend these results once more to the whole class of ARMA processes.
\vspace{0.25cm}
\newline
Of course, the next step will be to show asymptotic normality for $\hat{\gamma}_{LME}$ and $\hat{\sigma}_{LME}$. A very promising way to do so is to use and to extend the theory of tail array sums established by \cite{Rootzen1}, \cite{Rootzen2}, which both served as basis for Resnick \& St\^aric\^a's studies on asymptotic normality of Hill's estimator for autoregressive data (\cite{Resnick4}). Also, a simulation of the finite sample behavior of the LMEs would be very useful since the condition $n,k,n/k \to \infty$ is only fulfilled if one really collects a very large amount of data -- a well-known problem when using intermediate order statistics like $|X|_{n,n-k}$.

\newpage

\end{document}